\documentclass[12pt]{article} 

\usepackage{amsmath}
\usepackage{amsthm}
\usepackage{amsfonts}
\usepackage{mathrsfs}
\usepackage{stmaryrd}
\usepackage{setspace}
\usepackage{fullpage}
\usepackage{amssymb}
\usepackage{breqn}
\usepackage{enumitem}
\usepackage{bbold} 
\usepackage{authblk}
\usepackage{comment}
\usepackage{hyperref}
\usepackage{pgf,tikz}
\usepackage{graphicx}
\usepackage{subcaption}

\newtheorem*{thm*}{Theorem}
\newtheorem{thm}{Theorem}[section]%[chapter]

\newtheorem{conjecture}[thm]{Conjecture}

\newtheorem*{prop*}{Proposition}

\newcommand\ex{\ensuremath{\mathrm{ex}}}

\newcommand\cN{{\mathcal N}}

\newcommand{\ignore}[1]{}

\title{Paths are Tur\'an-good}
\author{Dániel Gerbner\\\small Alfr\'ed R\'enyi Institute of Mathematics\\
\small \texttt{gerbner.daniel@renyi.hu}}

\date{}

\begin{document}

\maketitle

\begin{abstract}
We show that among $K_{k+1}$-free $n$-vertex graphs, the Tur\'an graph contains the most copies of any path.
\end{abstract}

%\textbf{Keywords:} generalized Tur\'an,

%\smallskip

%\textbf{MSC 2020:} 05C35

\section{Introduction}

A fundamental theorem of extremal graph theory is due to Tur\'an \cite{T}, who showed that among $n$-vertex $K_{k+1}$-free graphs, the most edges are contained in the 
Tur\'an graph $T(n,k)$, which is the complete $k$-partite graph with each part of order $\lfloor n/k\rfloor$ or $\lceil n/k\rceil$. In general, the most number of edges is $n$-vertex $F$-free graphs is denoted by $\ex(n,F)$. Simonovits \cite{sim} proved that if $F$ has a color-critical edge, i.e., and edge whose deletion decreases the chromatic number of $F$, then for sufficiently large $n$ we have $\ex(n,F)=|E(T(n,k))|$.

In so-called \textit{generalized Tur\'an problems}, we count copies of a subgraph $H$ instead of edges. Let $\cN(H,G)$ denote the number of copies of $H$ in $G$. We denote by $\ex(n,H,F)$ the largest $\cN(H,G)$ among $n$-vertex $F$-free graphs $G$. The first result concerning $\ex(n,H,F)$ with $H\neq K_2$ is due to Zykov \cite{zykov}, who showed that $\ex(n,K_r,K_{k+1})=\cN(K_r,T(n,k))$.

One of the most studied question in this topic is the following: When is the Tur\'an graph extremal? It was systematically studied first by Gy\H ori, Pach and Simonovits \cite{gypl}. Gerbner and Palmer \cite{gp3} introduced the following definition. Given a graph $F$ with $\chi(F)=k+1$, we say that $H$ is \textit{$F$-Tur\'an-good} if $\ex(n,H,F)=\cN(H,T(n,k))$ for sufficiently large $n$. They made the following conjecture.

\begin{conjecture} \label{conji}
Each path is $K_{k+1}$-Tur\'an-good for every $k$.
\end{conjecture}

We denote by $P_\ell$ the path on $\ell$ vertices. We say that $H$ is a \textit{linear forest} if each connected component of $H$ is a path.
Results from \cite{gypl} imply that each path is $K_3$-Tur\'an-good and $P_3$ is $K_{k+1}$-Tur\'an-good for every $k$. Gerbner and Palmer \cite{gp3} showed that $P_3$ is $F$-Tur\'an-good for any $F$ with a color-critical edge and $\chi(F)>3$. This was extended by Gerbner \cite{gerb} to the case $\chi(F)=3$. He also showed that $P_4$ is $K_4$-Tur\'an-good and matchings are $K_{k+1}$-Tur\'an-good. Results of Gerbner \cite{gerbner2} imply that each path is $F$-Tur\'an-good for a large class of 3-chromatic graphs with color-critical edges, including odd cycles. Murphy and Nir \cite{mn} showed that $P_4$ is $K_{k+1}$-Tur\'an-good. Qian, Xie and Ge \cite{qxg} showed that $P_5$ is $K_{k+1}$-Tur\'an-good. Hei, Hou and Liu \cite{hhl} came close to prove Conjecture \ref{conji}. They showed that $P_\ell$ for $\ell\le 6$ is $F$-Tur\'an-good for any $F$ with a color-critical edge and any path is  $K_{k+1}$-Tur\'an-good for large enough $k$. Moreover, they reduced the problem in the remaining cases to complete multipartite graphs, as we will describe in Section 2.

We complete the proof of Conjecture \ref{conji}. We prove the following more general theorem.
 
\begin{thm}\label{main}
Let $F$ be a graph with a color-critical edge and $H$ be a linear forest. Then $H$ is $F$-Tur\'an-good.
\end{thm}

\section{Proof}

The main result of Hei, Hou and Liu \cite{hhl} is the following theorem.

\begin{thm}[Hei, Hou and Liu \cite{hhl}]
Let $F$ be a graph with a color-critical edge and $H$ be a graph with $\chi(H)<\chi(F)$. Let us assume that there is an $n$
-vertex $F$-free graph $G$ with $\ex(n,H,F)=\cN(H,G)$ that can be turned to a complete $k$-partite graph by adding and deleting $o(n^2)$ edges. If every complete $k$-partite graph $T$ has $\cN(H,T)\le \cN(H,T(n,k))$, then $H$ is $F$-Tur\'an-good.
\end{thm}

We remark that they stated a weaker theorem, adding additional conditions that hold for paths but are not actually used in the proof. Even in the weaker form, the theorem implies that a stability result accompanied with optimization on complete multipartite graphs is enough to complete the proof. They also proved the stability result for paths. In fact, they proved the following more general (and more usual) stability: if $G$ is an $n$
-vertex $F$-free graph with $\cN(P_\ell,G)=\ex(n,P_\ell,F)-o(n^\ell)$, then $G$ can be turned to a complete $k$-partite graph by adding and deleting $o(n^2)$ edges.

First we extend the stability result to linear forests. It follows from a result of Gerbner \cite{ge}. He showed that if stability holds for two graphs in the above sense, then the same holds for their vertex-disjoint union (in fact, he proved a stronger result: it is enough if one of the two graphs have the stability, and the other has the correct asymptotics). Therefore, Theorem \ref{main} follows from the next theorem.

\begin{thm}
Let $H$ be a linear forest. Then for any complete $k$-partite graph $G$ we have $\cN(H,G)\le \cN(H,T(n,k))$.
\end{thm}

\begin{proof}
First we prove the case $k=2$. We use induction on $|V(H)|$, the base cases $|V(H)|\le 2$ are trivial. Let us assume that the statement holds for smaller linear forests and prove it for $H$. Let $T$ be an $n$-vertex complete bipartite graph.

Assume first that $H$ has a component $K=P_{2i+1}$ on odd vertices. Then we let $H_1$ be the graph obtained by deleting an endpoint of $K$. By induction, the number of copies of $H_1$ is maximized by the Tur\'an graph. Clearly all the vertices of $T$ not in a copy of $H_1$ extend that $H_1$ to $H$ by adding them to the end of a $P_{2i}$ component of $H_1$, and each vertex is connected to exactly one endpoint of such a component. Let $H_1$ contain $x$ such components, then there are $x(n-|V(H_1)|)$ ways to extend $H_1$ to $H$. Each copy of $H$ is obtained the same number of times this way, thus the number of copies of $H$ is also maximized by the Tur\'an graph.

Assume now that $H$ has a component $K'=P_{2i}$. Then we let $H_2$ be the graph obtained by deleting the last edge of $K'$. By induction, the number of copies of $H_2$ is maximized by the Tur\'an graph, and so does the number of edges. We count copies of $H$ in $T$ by taking an edge $uv$, and a copy of $H_2$ on the remaining vertices, and then adding $uv$ to the end of a $P_{2i-2}$ component. The first two factors are maximized by the Tur\'an graph. Clearly, $u$ is adjacent to one endpoint of the $P_{2i-2}$ and $v$ is connected to the other endpoint, thus there are two ways to unite the $P_{2i-2}$ and $uv$ to a $P_i$ component. Each copy of $H$ is obtained the same number of times this way, completing the proof in the case $k=2$.

Let us continue with the case $k>2$.
Let $G$ be an $n$-vertex $F$-free complete $k$-partite graph with parts $A_1,\dots, A_k$. Let us assume that $|A_1|<|A_2|-1$. Then we move $\lfloor (|A_2|-|A_1|/2\rfloor$ vertices from $A_2$ to $A_1$ to obtain $G'$. We will show that the number of copies of $H$ does not decrease this way. Clearly, after repeating this enough times with all the parts, we obtain the Tur\'an graph.

A copy of $H$ intersects $A_1\cup A_2$ in a linear forest $H'$, and intersects $V(G)\setminus (A_1\cup A_2)$ in another linear forest $H''$. Observe that if $H'$ extends a copy of $H'$ inside $A_1\cup A_2$ to $H$, then the same way it extends any other copy of $H'$ inside $A_1\cup A_2$ to $H$. Therefore, it is enough to show that for any subforest $H'$ of $H$, the number of copies of $H'$ does not decrease when turning $G$ to $G'$. Observe that we turned $G[A_1\cup A_2]$ to $T(|A_1\cup A_2|,2)$, hence what we want is exactly the case $k=2$. We have already dealt with that case, thus we are done.
\end{proof}

\bigskip

\textbf{Funding}: Research supported by the National Research, Development and Innovation Office - NKFIH under the grants KH 130371, SNN 129364, FK 132060, and KKP-133819.

\end{document}